\documentclass{amsart}
\usepackage{amsfonts}
\usepackage{amsmath,amssymb}
\usepackage{amsthm}
\usepackage{amscd}
\usepackage{graphics}
\usepackage{graphicx}

\theoremstyle{remark}{
\newtheorem{Def}{{\rm Definition}}

\newtheorem{Rem}{{\rm Remark}}

}
\theoremstyle{plain}
{

\newtheorem{Prop}{Proposition}
\newtheorem{Thm}{Theorem}

}

\begin{document}
\title[Almost Morse functions of a certain class with natural construction]{Smooth functions which are Morse on preimages of values not being local extrema and constructing natural functions of the class on connected sums of manifolds admitting these functions}
\author{Naoki kitazawa}

\address{Osaka Central Advanced Mathematical Institute (OCAMI) \\
3-3-138 Sugimoto, Sumiyoshi-ku Osaka 558-8585
TEL: +81-6-6605-3103
}
\email{naokikitazawa.formath@gmail.com}
\urladdr{https://naokikitazawa.github.io/NaokiKitazawa.html}
\maketitle
\begin{abstract}
We discuss smooth functions which are Morse on preimages of values not being local extrema. We call such a function {\it internally Morse} or {\it I-Morse}. 

The {\it Reeb graph} of a smooth function is the space of all connected components of preimages of single points of it topologized with the natural quotient topology of the manifold and a vertex of it is a point corresponding to a preimage with critical points. A smooth function is {\it neat with respect to the Reeb graph} or {\it N-Reeb} if the preimages of the vertices are the closed subsets in the manifolds of the domains with interiors being empty.

We discuss I-Morse and N-Reeb functions, {\it IN-Morse-Reeb} functions. Our main result presents an IN-Morse-Reeb function respecting two such functions, on a connected sum of these given manifolds.

\end{abstract}
\section{Introduction.}
\label{sec:1}
{\it Morse}({\rm-Bott}) functions have been fundamental and strong tools in geometry of manifolds. From the viewpoint of singularity theory, they are also important mathematical (geometric) objects to study. 
For a function of a certain nice class, its {\it Reeb graph} can be defined. 
In short if the image of the set of all critical points of a smooth function on a closed manifold is discrete, then we can define its Reeb graph: see \cite[Theorem 3.1]{saeki4} for a general case and for more explicit case, see \cite{izar} for example. 
Reeb graphs of smooth functions are also classical strong tools in theory of Morse functions (\cite{reeb}).
These combinatorial objects are fundamental tools to obtain important information such as the 1st Betti numbers and the ranks of the fundamental groups. They are also important in visualizing Morse(-Bott) functions or certain generalized functions. We can also orient them naturally and have the {\it Reeb digraph} of the function.

We do not explain the definition of a Morse(-Bott) function rigorously. See \cite{milnor} for Morse functions and the well-known and fundamental natural correspondence between a so-called {\it $k$-handle}, a cornered smooth manifold represented as the product $D^k \times D^{m-k}$ of the $k$-dimensional unit disk $D^k$ in the $k$-dimensional Euclidean space ${\mathbb{R}}^k$ and the ($m-k$)-dimensional one $D^{m-k} \subset {\mathbb{R}}^{m-k}$ and a critical point of {\it index} $k$. For Morse-Bott functions, see \cite{bott} for example, natural generalizations of Morse functions. For systematic exposition on related singularity theory of differentiable maps, see \cite{golubitskyguillemin}.
 
In our paper, a smooth function is {\it internally Morse} or {\it I-Morse} if it is Morse at preimages of values not being local extrema. A smooth function is {\it neat with respect to the Reeb graph} if the Reeb graph is defined in such a way that the interior of the preimage of each preimage corresponding to each vertex is empty considered in the manifold of the domain. Such functions are explicitly discussed in \cite{lerariomeronizuddas}. Articles such as \cite{kitazawa5, kitazawa6} of the author are also closely related: construction of functions of an explicit subclass of this class is discussed. The preprint \cite{kitazawa9} of the author is also closely related.

We discuss I-Morse and N-Reeb functions: {\it IN-Morse-Reeb} functions or {\it IN-M-R} functions. 

Hereafter, a connected sum of two manifold is considered in the smooth category.

Our new result is as follows.
\begin{Thm}
\label{thm:1}
For two IN-M-R functions on closed and connected manifolds of dimension $m>1$, we have one on any connected sum of the two manifolds whose Reeb digraph is isomorphic to an arbitrary digraph $G_{\rm R}$ obtained in the following way.

Let $G_{{\rm R},1}$ and $G_{{\rm R},2}$ denote the Reeb digraphs of these two functions, respectively. From each Reeb digraph, we choose a point except a vertex corresponding to some local extremum. We identify these two points, define these points as the only one newly added vertex and have $G_{\rm R}$.  

\end{Thm}
For a graph $G$ and its vertex $v$, let ${\rm deg}_G(v)$ denote the degree of $v$ in $G$.
\begin{Def}
\label{def:1}
Let $G$ be a finite and connected graph admitting a piecewise smooth function which is smooth and injective on each edge and which has a local extremum only at a vertex of degree $1$. We regard this as a digraph naturally. Suppose that a closed and connected manifold $M$ of dimension $m>1$ admits an IN-M-R function as follows. We have the following. 

\begin{itemize}
\item The function has local extrema only at preimages corresponding to vertices of degree $1$.
\item The Reeb digraph can be embedded into the digraph $G$ topologically in such a way that the set of all vertices of degree $1$ are embedded into the set of all vertices of degree $1$ of $G$, that each vertex of degree $2$ is mapped to the interior of some edge of $G$, and that the orientation induced from $G$ is same as the orientation of the given orientation of the given Reeb digraph.
\end{itemize}
We call such a function a {\it $G$-IN-M-R} function.
If in addition, the number of critical points mapped by the natural quotient map to each vertex $v$ of degree at least $3$ of the Reeb digraph $G_{\rm R}$ of the $G$-IN-M-R function is ${\rm deg}_{G_{\rm R}}(v)-2$ for the IN-M-R function on $M$, 
and that of critical points mapped by the natural quotient map to each vertex of degree $2$ of the Reeb digraph $G_{\rm R}$ of it is $1$, 
then
we call such a function {\it $G$-simple} {\rm (}{\it $G$-S}{\rm )} function. 
\end{Def}
\begin{Thm}
\label{thm:2}
In Theorem \ref{thm:1}, if two given maps are $G_{{\rm R},i}$-S functions, then the resulting function can be also obtained as a $G_{\rm R}$-S function.  
\end{Thm}
\begin{Thm}
\label{thm:3}
Let $G$ be a finite and connected graph with no vertex of degree $2$ of 1st Betti number $0$ admitting a piecewise smooth function which is smooth and injective on each edge and which has a local extremum only at a vertex of degree $1$. We regard this as a digraph naturally.
We have the following.
\begin{enumerate}
\item \label{thm:3.1} Suppose that a closed and connected manifold $M$ of dimension $m>2$ admits a $G$-IN-R-M function.
Let ${W_f}_{\geq j}$ be the set of all vertices of the Reeb digraph $W_f$ of $f$ whose degrees are at least $j$.
Here, $M$ also admits an IN-M-R function whose Reeb digraph is isomorphic to a graph obtained by adding exactly ${\Sigma}_{v \in {W_f}_{\geq 3}} ({\rm deg}_{{W_f}_{\geq 3}}(v)-2)+{\Sigma}_{v \in {W_f}_{\geq 2}-{W_f}_{\geq 3}} 1$ new vertices of degree $2$ to $G$.
\item \label{thm:3.2} Suppose that there exist two $G$-IN-M-R functions and that their Reeb digraphs can be embedded into $G$
disjointly as in Definition \ref{def:1}. Let $M_1$ and $M_2$ denote the manifolds of the domains of these two given functions, respectively, and suppose that the dimensions are both $m>2$. Here, a manifold $M$ represented as a connected sum of $M_1$ and $M_2$ also admits an IN-M-R function whose Reeb digraph is isomorphic to a graph obtained by adding exactly ${\Sigma}_{v \in {W_f}_{\geq 3}} ({\rm deg}_{{W_f}_{\geq 3}}(v)-2)+{\Sigma}_{v \in {W_f}_{\geq 2}-{W_f}_{\geq 3}} 1$ new vertices of degree $2$ to $G$.
\end{enumerate}
\end{Thm}
They also give a kind of new results on the following problem. For a class ${\mathcal{M}}_{\rm s}$ of smooth maps, consider two maps of the class on some closed and connected manifolds. On a connected sum, can we construct a map of the class respecting topological properties and combinatorial ones of the given maps? 

This can be discussed easily in the class of {\it fold} maps, higher dimensional versions of Morse functions. See \cite{golubitskyguillemin} for related fundamental singularity theory and see also a pioneering explicit differential topological study \cite{saeki1}.

A fold map is a smooth map locally represented as the product map of a Morse function and the identity map on some smooth manifold. We can define the index of a singular point of a fold map uniquely. We need to understand the notion of fold map and it is sufficient to know this, in understanding some arguments of our proof of Theorems \ref{thm:1}, \ref{thm:2} and \ref{thm:3} with FIGUREs \ref{fig:1}, \ref{fig:2}, \ref{fig:3}, \ref{fig:4}, and \ref{fig:5}. 

We go back to the problem.  {\it Special generic} maps are generalizations of the canonical projections of the unit spheres $S^k \subset {\mathbb{R}}^{k+1}$ in the sense that singular points are higher dimensional versions of critical points of index $0$ for Morse functions and form a subclass of the class of fold maps: \cite{burletderham, furuyaporto, saeki2} are related pioneering studies.

This is also shown in some specific cases for {\it round} fold maps, the images of the set of all singular points of which are concentric spheres and which have been introduced first by the author as another subclass of fold maps generalizing the canonical projections of the unit spheres: see \cite{kitazawa1, kitazawa2, kitazawa3, kitazawa4}. 

The next section proves our new result.

\section{On our main result.}

We shortly present fundamental methods and theory.

We need to apply the well-known natural correspondence of $k$-handles and critical points whose indices are $k$ in several scenes for $k \geq 1$.

Let $\mathbb{R}:={\mathbb{R}}^1$. For $a_1, a_2 \in \mathbb{R}$ with $a_1<a_2$, let $[a_1,a_2]:=\{t \mid a_1 \leq t \leq a_2\}$.
Let $a<s<b$ be real numbers.

Hereafter, the boundary of a topological manifold $X$ is denoted by $\partial X$. For example, $\partial D^m=S^{m-1}$. 

Let $a<s<b$ be real numbers. Hereafter, let $f_{a,b,s}:M_{a,b,s} \rightarrow \mathbb{R}$ denote a Morse function on an $m$-dimensional compact and connected manifold $M_{a,b,s}$ ($m>1$) whose image is $[a,b] \subset \mathbb{R}$, whose unique critical value is $s$ and the boundary of the manifold of the domain is the disjoint union of the preimage ${f_{a,b,s}}^{-1}(a)$ of $a$ and the preimage ${f_{a,b,s}}^{-1}(b)$ of $b$. On the preimage ${f_{a,b,s}}^{-1}(s-\epsilon) \subset {f_{a,b,s}}^{-1}([a,s-\epsilon])$, handles are attached disjointly along $\partial D^k \times D^{m-k}$ to obtain $M_{a,b,s}$ where $\epsilon>0$ is a sufficiently small number.
 
Another important theorem is Ehresmann's fibration theorem (of the relative version). In several scenes, we suitably find product bundles with fibers diffeomorphic to $D^{m-1}$ in the $m$-dimensional manifolds of the domains over closed intervals in the images of our smooth functions whose projections are given by the restrictions of the original smooth functions and which contain no critical point of the original functions. See \cite{ehresmann} for our original study and see \cite{saeki3} for explicit usage of the relative version in determining types of preimages around critical values of Morse functions (on compact surfaces).

We also discuss such a trivial smooth bundle in the preprint \cite{kitazawa7}, where we do not need to understand the arguments related to the preprint well.

Related to Ehresmann' theorem, ${f_{a,b,s}}^{-1}([a,s-\epsilon])$ is regarded as the total space of a smooth bundle over $[a,s-\epsilon]$ whose fiber is diffeomorphic to ${f_{a,b,s}}^{-1}(a)$. The projection is realized by the original Morse function of course. For ${f_{a,b,s}}^{-1}([s+\epsilon,b])$, we can discuss similarly.

For such a Morse function on an $m$-dimensional compact and connected manifold whose boundary is not empty, we can define its Reeb (di)graph naturally. As vertices, points corresponding to connected components of the preimages of local extrema are added and the degrees are always $1$ for these new vertices.

Hereafter, for a smooth function $c:X \rightarrow \mathbb{R}$ whose Reeb digraph can be defined, its Reeb digraph is denoted by $W_c$. 
Let $q_c:X \rightarrow W_c$ denote the quotient map. We can define a continuous function $\bar{c}:W_c \rightarrow \mathbb{R}$ with $c=\bar{c} \circ q_c$ uniquely.
We can also observe the following in our situation.

\begin{Prop}
\label{prop:1}
We choose an arbitrary edge $e_a$ {\rm (}$e_b${\rm )} of the Reeb digraph $W_{f_{a,b,s}}$ mapped onto $[a,s]$ {\rm (}$resp. [s,b]${\rm )}.
We can find a product bundle $B_{a,b,s}$ over $[a,b]$ with its fiber diffeomorphic to $D^{m-1}$ as before smoothly embedded in the the $m$-dimensional manifold $M_{a,b,s}$ which contains no critical point of the original function $f_{a,b,s}$ and which is, by the quotient map, mapped onto the closure $e_{a,b}$ of the union $e_a \bigcup e_{a \mapsto b}$ {\rm (}resp. $e_{b \mapsto a} \bigcup e_b${\rm )} of any pair of distinct two edges of the Reeb digraph $W_{f_{a,b,s}}$ mapped onto the image $[a,b]$ by the function $\bar{f_{a,b,s}}$. Note that the closure of the union of the edges is considered in the Reeb digraph $W_{f_{a,b,s}}$.
\end{Prop}
This argument has been also used in \cite{kitazawa6}. See also the preprint \cite{kitazawa7}.
\begin{Rem}
\label{rem:1}
A variant of Proposition \ref{prop:1} is also shown in the case $f_{a,b,s}$ is a function with no critical point. In this case, the preimage ${f_{a,b,s}}^{-1}([a,b])$ is regarded as the total space of a trivial smooth bundle over $[a,b]$ whose fiber is diffeomorphic to both ${f_{a,b,s}}^{-1}(a)$ and ${f_{a,b,s}}^{-1}(b)$ and diffeomorphic to an ($m-1$)-dimensional closed and connected manifold. The projection is given by the function $f_{a,b,s}$ of course. The Reeb digraph is a graph with exactly one edge with two vertices.
\end{Rem}

\begin{Prop}
\label{prop:2}
In Definition \ref{def:1}, let $m>2$ and let $f:M \rightarrow \mathbb{R}$ denote the IN-M-R function.
Choose the Reeb digraph $W_f$ embedded into $G$. We remove the interior of a small regular neighborhood of each vertex of degree $1$ of the Reeb digraph and this new digraph with the induced orientation is denoted by $W_{f,G}$. 
We can find an $m$-dimensional smooth compact and connected manifold $B_{f,G} \subset M$ in the $m$-dimensional manifold $M$ of the domain which contains no critical point of the original function $f$, which is, by the quotient map $q_f$, mapped onto the refined digraph $W_{f,G}$, and collapses to a digraph homeomorphic to this naturally, and which enjoys the following. 
\begin{enumerate}
\item \label{prop:2.1}
The restriction of $f$ to $B_{f,G} \subset M$ gives the structure of a trivial bundle over $I_{A_{f,G},j}$ whose fiber is diffeomorphic to the disjoint union of finitely many copies of $D^{m-1}$ for any open interval $I_{A_{f,G},j}$ obtained as a connected component of $I_{f,G}-N(A_{f,G})$ we can define in the following way.
\begin{itemize}
\item We define $I_{f,G}:=f(B_{f,G})$ as the image of the function and it is diffeomorphic to a closed interval.
\item A subset $A_{f,G}$ in the interior of $I_{f,G}$ is defined as the image of the set of all vertices whose degrees are at least $3$ by the function $\bar{f}:W_f \rightarrow \mathbb{R}$ with $f=\bar{f} \circ q_f$. The set $N(A_{f,G}) \subset \mathbb{R}$ is defined as a suitable sufficiently small regular neighborhood of $A_{f,G}$.
\item The number of connected components of the preimage of $f^{-1}(I_{A_{f,G},j})$ and that of $B_{f,G} \bigcap f^{-1}(I_{A_{f,G},j})$ agree for any interval $I_{A_{f,G},j}${\rm :} the bundle $B_{f,G} \bigcap f^{-1}(I_{A_{f,G},j})$ is regarded as a subbundle of the bundle $f^{-1}(I_{A_{f,G},j})$, whose projection is given by the restriction of the function $f$ to $f^{-1}(I_{A_{f,G},j})$.
\end{itemize}
\item \label{prop:2.2}
 This is related to local structures of the restriction of $f$ to $B_{f,G} \subset M$ {\rm (}around $A_{f,G}${\rm )}.
\begin{enumerate}
\item \label{prop:2.2.1}
Each connected component of the preimage $f^{-1}(p)$ {\rm (}$p \in N(A_{f,G})${\rm )} is a disjoint union of finitely many copies of the disk $D^{m-1}$ or represented as a bouquet of copies of the disk $D^{m-1}$. More precisely, 
this is a non-manifold space only if $p$ is a boundary point of $N(A_{f,G})$.

In the case of a bouquet, which is not a manifold, this is obtained as follows.
\begin{itemize}
\item We choose two copies of the disk, identify them by choosing one point in the boundary of each disk, in the case of a bouquet of exactly two disks.
\item In the case of a bouquet of at least three disks, we choose a point which is not in the interior and which is not a so-called non-manifold point for the existing bouquet and another point in the boundary of another new disk, identify them, and we do this one after another, inductively.
\end{itemize}  
\item \label{prop:2.2.2}
Let $I_{{\rm o},f,G}$ denote a closed interval in the interior of $I_{f,G}$ such that the complementary set $I_{f,G}-I_{{\rm o},f,G}$ is the disjoint union of two sufficiently small closed intervals.
The restriction of $f$ to the intersection of the boundary $\partial B_{f,G} \subset M$ and the preimage $f^{-1}(I_{{\rm o},f,G})$ is a Morse function such that preimages of single points containing no critical point are disjoint unions of copies of $S^{m-2}$.
Furthermore, suppose that $p$ and $q$ are two boundary points of a connected component of $N(A_{f,G}) \subset \mathbb{R}$ with $p<q$ and that for a connected component $W_{f,[p,q],j}$ of the preimage ${\bar{f}}^{-1}([p,q])$ having exactly one vertex $v_{[p,q],j}$, the number of edges entering the vertex and that of edges departing from the vertex are $N_p$ and $N_q$, respectively. In this case, the preimage $f^{-1}(p)$ contains exactly $N_q-1$ critical points and the preimage $f^{-1}(q)$ contains exactly $N_p-1$ critical points of the Morse function.    
All critical points of the function appear in this way. 
\item \label{prop:2.2.3}
The intersection of each connected component of $f^{-1}(p)$ {\rm (}$p \in I_{{\rm o},f,G}${\rm )} and $M-B_{f,G}$ is also connected.
\end{enumerate}
The class of Morse functions like the restriction of $f$ to the intersection of the boundary $\partial B_{f,G} \subset M$ and the preimage $f^{-1}(I_{{\rm o},f,G})$ is discussed in \cite{michalak1, saeki4} and
see also \cite{saekisuzuoka} {\rm (}\cite[Figure 4]{saekisuzuoka}{\rm )} and \cite{kitazawa2, kitazawa3} for example.
The restriction of $f$ to $f^{-1}(I_{{\rm o},f,G})$ is also regarded as a so-called {\it non-singular extension} of this Morse function and see also \cite{curley, iwakura} for example.
\item \label{prop:2.3}
In the situation of Theorem \ref{thm:2}, the $m$-dimensional smooth submanifold in $M$ is also diffeomorphic to $D^m$ after the corner is smoothed.
\end{enumerate}
\end{Prop}

\begin{proof}
The main ingredient in our proof is local construction around each connected component of $f^{-1}(N(A_{f,G}))$
and around $f^{-1}(I_{f,G}-N(A_{f,G}))$.

We apply Proposition \ref{prop:1} with Remark \ref{rem:1} for $f^{-1}(N(A_{f,G}))$. We also obtain such product bundles as possible for pairs of edges as presented in Proposition \ref{prop:1} (Remark \ref{rem:1}). Due to the condition $m>2$ on the dimension of the manifold, we can choose them disjointly.
Around $f^{-1}(I_{f,G}-N(A_{f,G}))$, we can have a situation essentially same as that of Remark \ref{rem:1}. We can have a product bundle as presented in Proposition \ref{prop:1} with Remark \ref{rem:1}.

We deform and glue the local product bundles suitably to have our desired manifold $B_{f,G}$. We also remove some of the product bundles we do not need. We can have a situation such that each critical point of the function of the restriction of $f$ to the intersection of the boundary $\partial B_{f,G} \subset M$ and the preimage $f^{-1}(I_{{\rm o},f,G})$ corresponds to a $1$-handle, which is used to connect two connected components of $f^{-1}(a)$, or an ($m-1$)-handle, which is used to decompose a single connected component into two connected summands, and have connected components of $f^{-1}(b)$. Note that here we respect the local situation as in Proposition \ref{prop:2} and abuse the notation of Proposition \ref{prop:2}.  

\begin{figure}
	\includegraphics[width=80mm,height=80mm]{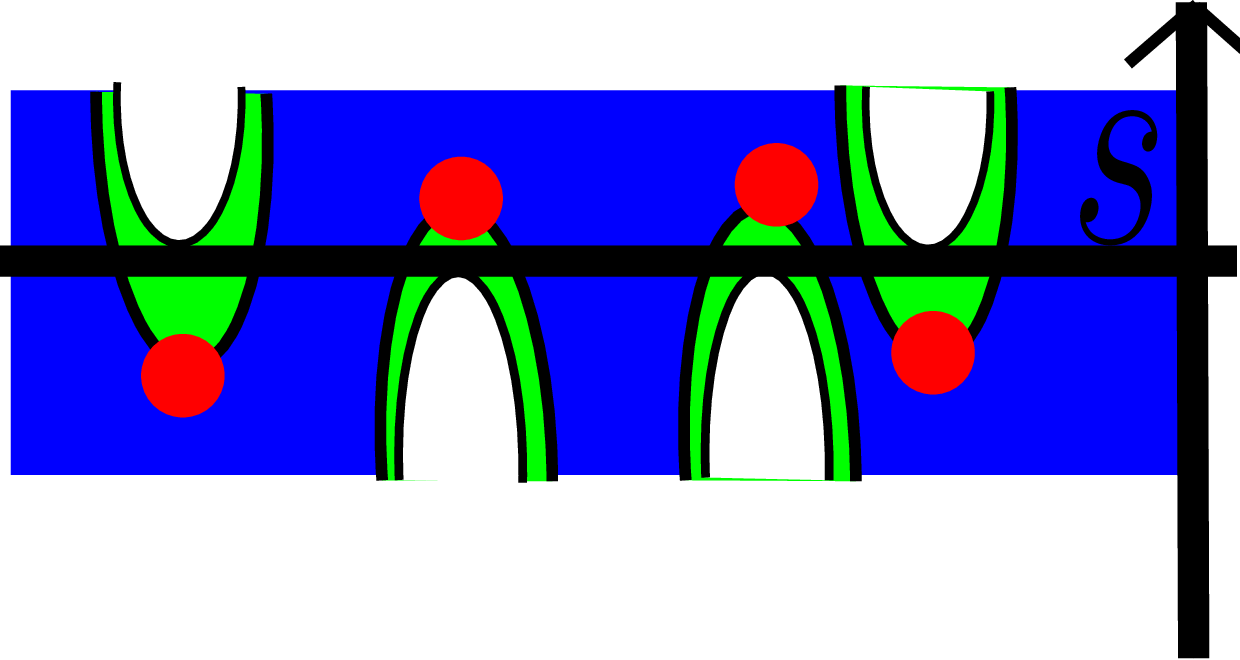}
	\caption{A part of the manifold $M$ mapped to a vertex $v_{[p,q,],j}$ by the original quotient map $q_f$ and (a neighborhood of) $[p,q] \in \mathbb{R}$ by the original function $f$ in Proposition \ref{prop:2}.
		The manifold $B_{f,G} \subset M$
		is shown in blue and the (interior of the) manifold $M-B_{f,G}$ is colored in green. Furthermore, red dots show critical points of the Morse function such that preimages of single points containing no critical point are disjoint unions of copies of $S^{m-1}$.}
	\label{fig:1}
\end{figure}
 This completes the proof.
\end{proof}
\begin{Rem}
\label{rem:2}
In Proposition \ref{prop:2}, for a connected subgraph $W_{{\rm S},f}$ of the Reeb digraph $W_f$ with the set of all vertices of degree $1$ of $W_{{\rm S},f}$ being mapped into the set of all vertices of degree $1$ of $G$, each vertex of degree $2$ being mapped to the interior of some edge of $G$, and the orientation induced from $G$ being same as the orientation of the given orientation of the given digraph $W_{{\rm S},f} \subset W_f$, we can discuss similarly to the original case. This is also important in our proof of Theorem \ref{thm:2}. 
\end{Rem}

\begin{Rem}
	\label{rem:3}
We can prove Proposition \ref{prop:2} with Remark \ref{rem:2} in the case $m=2$ in some specific cases. Due to the low dimensional situation, our arguments are restricted. However, it is sufficient to consider the case $m>2$. See also Remark \ref{rem:4}.
\end{Rem}
For deformations of Morse functions and their Reeb graphs, see also \cite{marzantowiczmichalak, michalak1, michalak2} for example.
\begin{proof}[Our proof of Theorems \ref{thm:1} and \ref{thm:2}]
Theorem \ref{thm:1} is shown as in FIGUREs \ref{fig:2} and \ref{fig:3}. We give more precise exposition on this.

For a chosen point in each of the Reeb digraphs, we consider small segments homeomorphic to a closed interval in the Reeb digraphs and contain the chosen points in their interiors. They are seen as segments $e_{a,b}$ as presented in Proposition \ref{prop:1} or Remark \ref{rem:1}.
We encounter the situation of Proposition \ref{prop:1} with Remark \ref{rem:1} explicitly. We respect the notation and the situation and scale the images of these local functions to $[a,b]$.
We choose the product bundles in the manifolds of the domains of the local Morse functions mapped onto $e_{a,b}$ by the quotient map $q_{f_{a,b,s}}$. We remove the interiors of the product bundles. We glue the remaining functions (onto $[a,b]$) preserving the value at each point as we do in \cite{kitazawa6} and the preprint \cite{kitazawa7}. 

We have a new IN-M-R function and its digraph where the two chosen segments are attached respecting the original orientations.

By considering another local smooth function, we can see that we can deform the function to deform the Reeb digraph into a desired one. As a kind of extra remarks, we refer to \cite[Figure 4 (6)]{michalak2} and the inverse operation of the modification of the Reeb digraph there does not work in general where in our present case we can apply the inverse operation. Related more precise arguments are explained in FIGUREs \ref{fig:2} and \ref{fig:3}.
\begin{figure}
	\includegraphics[width=80mm,height=80mm]{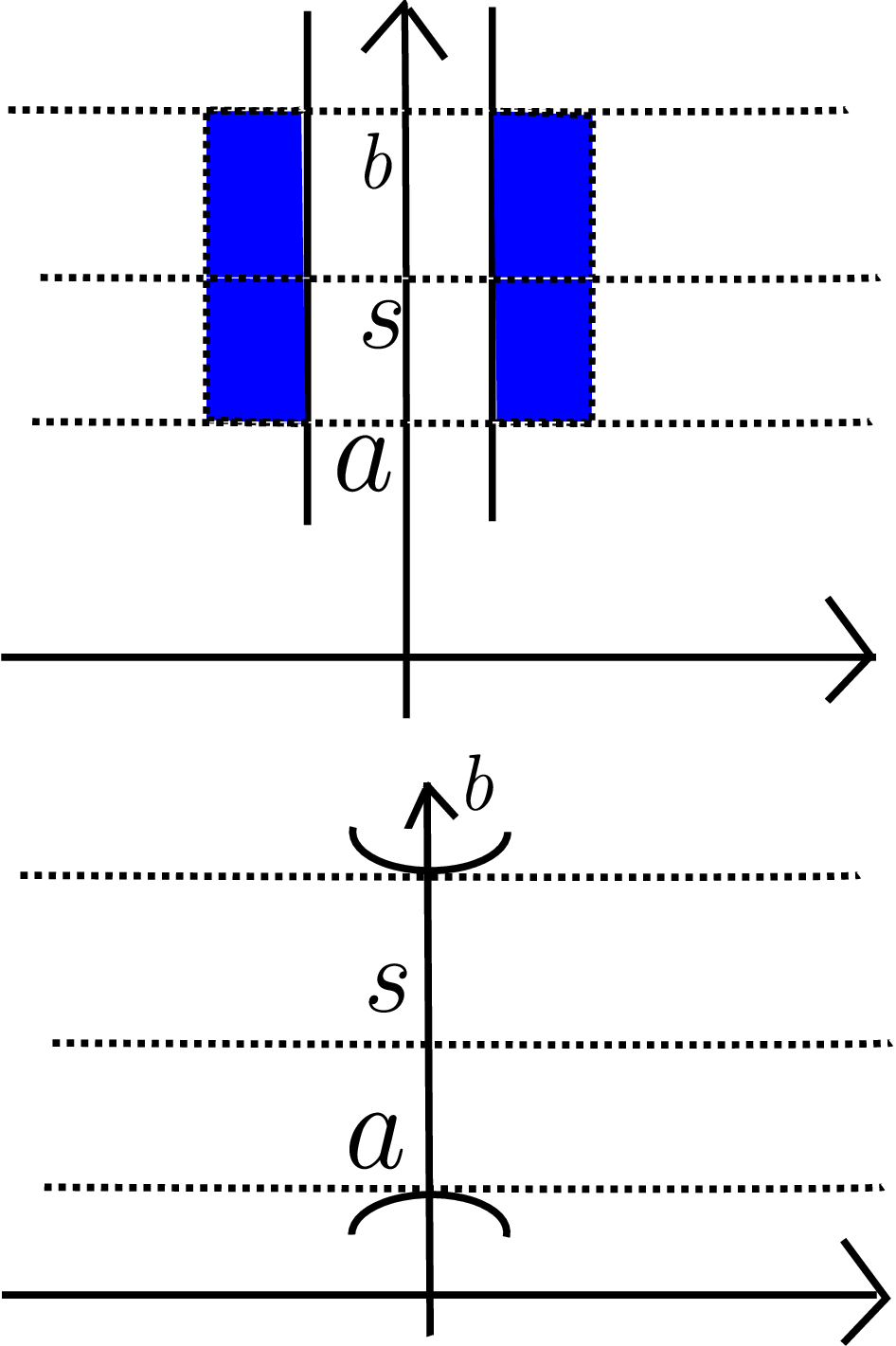}
	\caption{Around each of the two product bundles, the product map of the original function $f_{a,b,s}$ and another suitable function is considered and regarded as a fold map whose singular points are regarded as the higher dimensional versions of critical points of index $0$. The (preimages of the) blue colored regions show the product bundles and we remove. After that, we glue the remaining functions (onto $[a,b]$) preserving the value at each point as we do in \cite{kitazawa6} and the preprint \cite{kitazawa7}. We also consider the smoothing.}
	\label{fig:2}
\end{figure}
\begin{figure}
	\includegraphics[width=80mm,height=80mm]{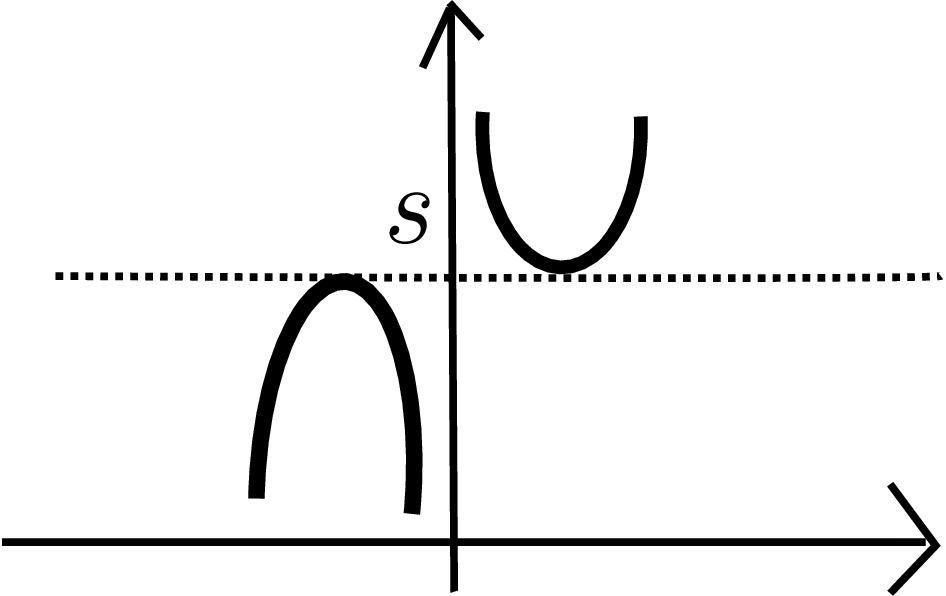}
	\caption{We deform the resulting local fold map of FIGURE \ref{fig:2} to have our desired local Morse function.}
	\label{fig:3}
\end{figure}

The resulting manifold of the domain is a desired connected sum. We can exchange the diffeomorphism used for gluing the local functions preserving the values of the functions from the orientation preserving (reversing) one to another reversing (resp. preserving) one.

This completes the proof of Theorem \ref{thm:1}.

By our definitions and construction, we have Theorem \ref{thm:2}.

We prove Theorem \ref{thm:3}.

We prove Theorem \ref{thm:3} (\ref{thm:3.1}). We encounter the situation of Proposition \ref{prop:2} with Remark \ref{rem:2} explicitly. 
We can apply Proposition \ref{prop:2} for the function $f$.

Several exercises on Morse functions and Reeb graphs appear. First, thanks to higher dimensional versions of \cite{gelbukh1, gelbukh2, michalak1}, we have a $G$-S Morse function $f_{G,S^m}$ such that the preimages of points containing no critical point are disjoint unions of the copy of $S^{m-1}$ and that the Reeb digraph is isomorphic to $G$. 
$1$-handles and ($m-1$)-handles are corresponded to these singular points.
See also the preprints \cite{kitazawa8, kitazawa9} of the author, where we do not assume related knowledge.

For the function $f_{G,S^m}$, we apply Remark \ref{rem:2} by $W_{{\rm S},f_{G,S^m}}:=W_f$ in Theorem \ref{thm:2}. Beforehand, we need to construct $f_{G,S^m}$ suitably to continue our arguments. For this we need to attach handles in such a way that the following hold. 
\begin{itemize}
\item We change the preimage of a single point of an edge of $W_f$ to a manifold diffeomorphic to $S^{m-n}$ and attach $1$-handles and $n-1$ handles suitably respecting Proposition \ref{prop:2}.
\item We attach some handles in the way same as the way where we attach handles for Proposition \ref{prop:2} with $f$ and $W_f$. More precisely, corresponding handles are attached to connected components or unions of connected components same as corresponding connected components in the case of $f$ and $W_f$.  
\end{itemize}

We can do this and we also have a corresponding $G$-S function. As we do in the proof of Theorem \ref{thm:1}, we find an $m$-dimensional smooth compact and connected manifold in the original manifold $M$ and an $m$-dimensional smooth compact and connected manifold in the original manifold $S^m$ mapped onto the same refined digraph $W_{f,G}$. As we do in the proof of Theorem \ref{thm:1}, we can remove these $m$-dimensional submanifolds, diffeomorphic to $D^m$, as presented in Proposition \ref{prop:2} (\ref{prop:2.3}), and we can glue the remaining functions preserving the value at each point. 
  
Related to this, we go back to FIGURE \ref{fig:1}. This shows Proposition \ref{prop:2} with the manifold $B_{f,G} \subset M$
being shown in blue and the (interior of the) manifold $M-B_{f,G}$ being colored in green. 
More precisely, this shows a part mapped to a vertex by the quotient map $q_f$ and mapped to $[a,b] \subset \mathbb{R}$ by $f$.
Furthermore, red dots also show points which are critical points of the resulting function, obtained after we glue the local functions. The index of each critical point of the resulting function is $1$ or $m-1$.

(After a slight deformation,) we can see that the resulting Reeb digraph is, around each vertex of degree $1$ of $W_{f}$, embedded in $G$, as presented in FIGURE \ref{fig:4}.

\begin{figure}
	\includegraphics[width=80mm,height=80mm]{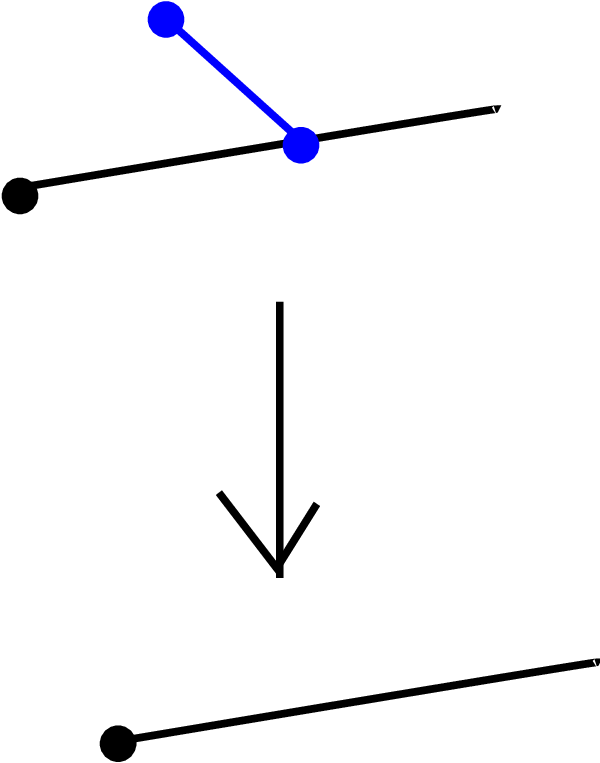}
	\caption{The Reeb digraph $W_f$ around a vertex $v$ of degree $1$ in $W_f$, mapped to a vertex of degree $1$ of $G$. The blue colored edge is originally a part of an edge of the graph $G$ containing the vertex $v \in W_f \subset G$. To each of blue colored two vertices, exactly one critical point of the local Morse function is mapped. Last, we can deform the function with the Reeb digraph by applying a well-known argument of canceling a pair of handles.}
	\label{fig:4}
\end{figure}

For this, more precisely, as we do in the proof of Theorem \ref{thm:1}, we need to consider a local product map as in FIGURE \ref{fig:3} and deform the resulting local fold map as in FIGURE \ref{fig:3}. Check FIGURE \ref{fig:5} with the attached exposition.

\begin{figure}
	\includegraphics[width=80mm,height=80mm]{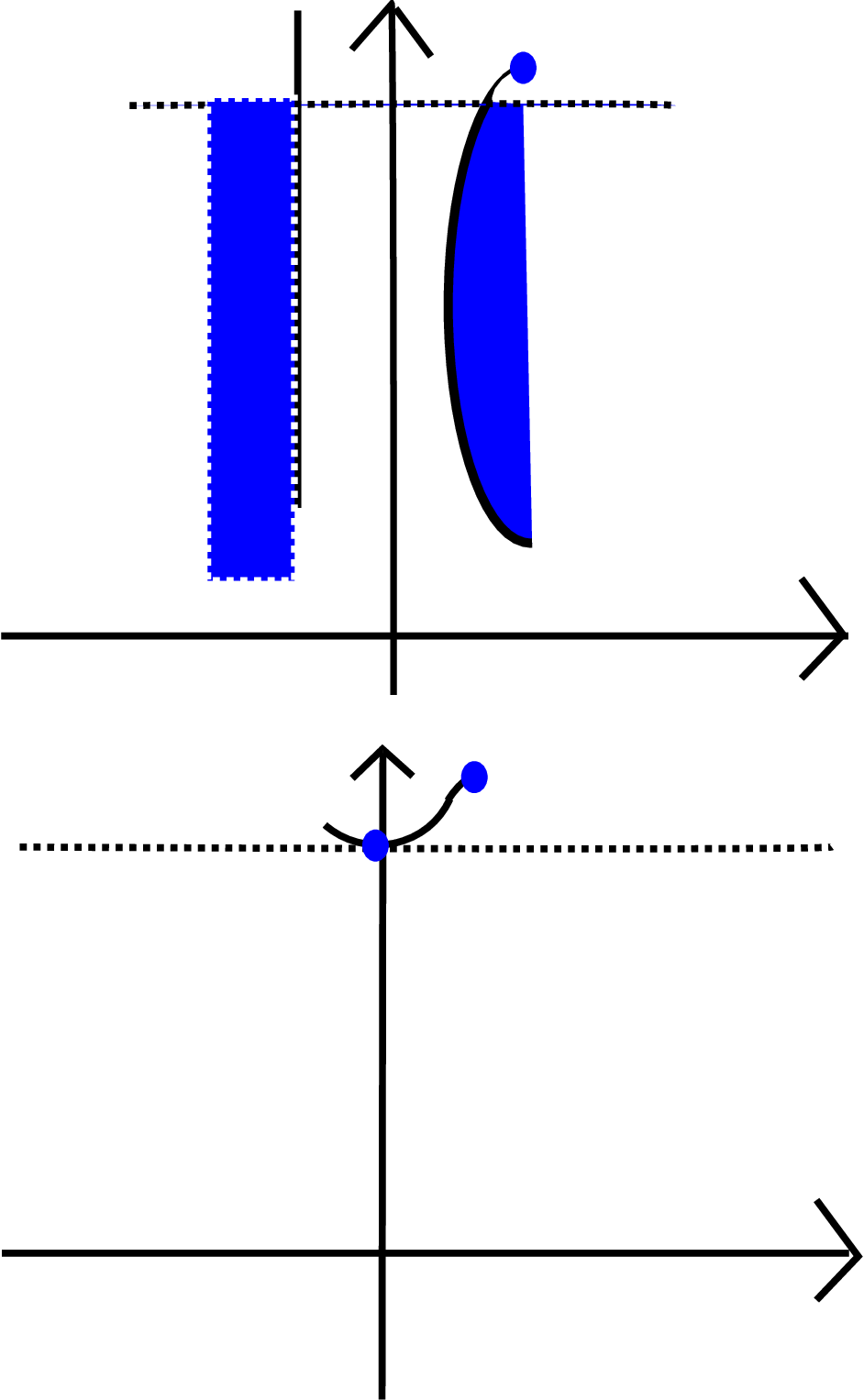}
	\caption{Around each of the two product bundles around the vertex $v$ of degree $1$ of $W_f$, mapped into $G$. The product map of the original local function and another suitable function is considered in the upper figure and regarded as a fold map whose singular points are regarded as the higher dimensional versions of critical points of index $0$. Similar arguments are also in FIGURE \ref{fig:2}. The (preimages of the) blue colored regions in the upper figure show the product bundles and we remove their interiors as we do in FIGURE \ref{fig:2}. The blue dotted points show (the points whose preimages considered for the local fold map contain) critical points of the local Morse function presented in the upper part of FIGURE \ref{fig:4} by the Reeb graph. As presented in FIGURE \ref{fig:3}, the resulting local fold map is deformed to have our desired situation of FIGURE \ref{fig:4} (the lower figure here). The critical points shown in blue are mapped to the blue colored vertices of the graph of (the upper figure of) FIGURE \ref{fig:4}.}
	\label{fig:5}
\end{figure}
FIGURE \ref{fig:4} also shows the final deformation of the function and its Reeb digraph to have a desired function with a desired Reeb digraph by canceling a pair of critical points and a pair of corresponding handles. This is also presented in \cite[Figure 5 (8)]{michalak2}.
For the structure of the resulting Reeb digraph, we also respect Proposition \ref{prop:2} (\ref{prop:2.1}, \ref{prop:2.2}), especially, (\ref{prop:2.2.3}), with Remark \ref{rem:2}.



As Theorem \ref{thm:1}, we can exchange the diffeomorphism for gluing the functions from the orientation preserving (reversing) one to another orientation reversing (resp. preserving) one easily.

This completes our proof of Theorem \ref{thm:3} (\ref{thm:3.1}).

By the condition and the argument above, we can prove Theorem \ref{thm:3} (\ref{thm:3.2}) similarly.

This completes the proof.

\end{proof}

\begin{Rem}
\label{rem:4}
	Related to Remark \ref{rem:3}, for example, in the case $m \geq 2$, if $f$ is $W_f$-S simple and the index of each vertex of $W_f$ is $1$, $2$ or $3$, then we can show Proposition \ref{prop:2} affirmatively. We can also apply Theorem \ref{thm:2} for such $f$.
\end{Rem}
\begin{Rem}
\label{rem:5}
	In Theorem \ref{thm:2}, let us weaken the condition on the embedding of vertices of degree $2$ of the Reeb digraph $W_f \subset G$ and consider an embedding which may map vertices of $W_f$ of degree $2$ into the vertex set of graph $G$. Let ${V_{{G,W_f}=2}} \subset {W_f}_{\geq 2}-{W_f}_{\geq 3}$ denote the set of all vertices of $W_f$ of degree $2$ mapped into the vertex set of $G$. Then the number of the Reeb digraph of the resulting function of $M$ added newly to $G$ is changed to ${\Sigma}_{v \in {W_f}_{\geq 3}} ({\rm deg}_{{W_f}_{\geq 3}}(v)-2)+{\Sigma}_{v \in {W_f}_{\geq 2}-({W_f}_{\geq 3} \sqcup {V_{{G,W_f}=2}})} 1$ from ${\Sigma}_{v \in {W_f}_{\geq 3}} ({\rm deg}_{{W_f}_{\geq 3}}(v)-2)+{\Sigma}_{v \in {W_f}_{\geq 2}-{W_f}_{\geq 3}} 1$ for each of Theorem \ref{thm:3} (\ref{thm:3.1}) and (\ref{thm:3.2}). 
	\end{Rem}

The following is a kind of our remark for our future studies.
\begin{Rem}
\label{rem:6}
	As a future work, we can consider classifications of IN-M-R functions of certain classes via Reeb digraphs, respecting classifications of certain classes of Morse-Bott functions \cite{gelbukh1, gelbukh2, martinezalfaromezasarmientooliveira1, martinezalfaromezasarmientooliveira2, marzantowiczmichalak,michalak1, michalak2}, for example. For this, as a kind of explicit preparation, we need to find additional new methods of construction of IN-M-R functions other than our present new result.

It is also important to find explicit motivations to study functions of certain classes generalizing the classes of Morse-Bott functions and formulate such classes. Morse-Bott functions are fundamental and important in topology and differential topology of manifolds and some symplectic topology such as theory of symplectic toric manifolds and moments maps, composing which with canonical projections gives Morse-Bott functions. Are these new classes important in certain fields of geometry?
	\end{Rem}

\section{Acknowledgment.}
The author would like to thank Antonio Lerario for the information \cite{lerariomeronizuddas}. Thanks to this, the author has been motivated to produce the present study.
\section{Conflict of interest and data.}
The author is a researcher at Osaka Central
Advanced Mathematical Institute (OCAMI researcher). The institute is supported by MEXT Promotion of Distinctive Joint Research Center Program JPMXP0723833165. The author is not employed there, where the author thanks the institute.

During the preparation of (the present version of) the present paper, the author informally joined the project Joint Research Center for Advanced and Fundamental Mathematics-for-Industry 2025a028 
(https://joint.imi.kyushu-u.ac.jp/research-reports/year-2025/) and discussed singularity theory of differentiable maps and related differential topology of manifolds, and related applications to multi-optimization problems, especially, geometric structures of the problems. Related to this project, the author has also discussed our present paper and this has helped the author to improve the quality of the paper. 

The author would like to thank Naoki Hamada for discussions on Reeb graphs, reconstruction of nice smooth functions with Reeb graphs isomorphic to given graphs and applications to optimization problems. 
The author would like to thank Kouki Iwakura for discussions related to non-singular extensions with \cite{iwakura}. The author would also like to thank Takahiro Yamamoto for comments from the viewpoint of singularity theory and application to geometry and ones related to Remark \ref{rem:6}. 
The author would like to thank all members related to the project for important comments again.

Other than the present file, no data are associated.

\end{document}